\documentclass{amsart}%
\usepackage{amssymb}
\usepackage{amsfonts}
\usepackage{amsmath}
\usepackage{graphicx}%
\providecommand{\U}[1]{\protect\rule{.1in}{.1in}}
\newtheorem{theorem}{Theorem}[section]
\theoremstyle{plain}

\newtheorem{corollary}[theorem]{Corollary}
\newtheorem{definition}[theorem]{Definition}

\newtheorem{proposition}[theorem]{Proposition}

\numberwithin{equation}{section}
\begin{document}
\title[Model Spaces and Pick Spaces]{Characterizing model spaces among the finite dimensional RKHS with Pick kernels}
\author{Richard Rochberg}
\maketitle

\begin{abstract}
We give several characterizations of those finite dimensional HSRK with
complete Pick kernels which are model spaces. One characterization involves
the size of the solution to a multiplier problem. Another involves having a
conjugation operator which is compatible with the RKHS structure.

\end{abstract}
\date{}

\section{Introduction and Summary}

We begin with an informal overview, detailed definitions and statements are in
the later sections.

Let $\mathcal{F}$ be the collection for finite dimensional reproducing kernel
Hilbert spaces (RKHS) with irreducible complete Pick kernels. Those are
exactly the spaces which are rescalings of spaces generated by finite sets of
Dirichlet-Arveson reproducing kernels; the exact statement is Definition
\ref{pick} below.

We interested in characterizing those $H\in\mathcal{F}$ that arise from finite
dimensional simple model spaces. That is, let $H^{2}$ be the the classical
Hardy space and $B\in H^{2}$ a finite Blaschke product with simple zeros at
the points $X=\left\{  x_{i}\right\}  _{i=1}^{n}$ in the complex unit disk
$\mathbb{B}^{1}.$ Let $K_{B}$ be the finite dimensional Hilbert space%

\[
K_{B}=H^{2}\ominus BH^{2}.
\]
We regard $K_{B}$ as an RKHS by declaring the functionals of evaluation at
points of $X$ to be the reproducing kernels; that is, the kernels are the
$k_{i}\in K_{B},$ $i=1,...,n,$ which satisfy $\left\langle f,k_{i}%
\right\rangle =f(x_{i})$ for all $f\in K_{B}.$ We call those RKHS model spaces
and call their rescalings $r-$model spaces. We denote the collection of all
$r-$model spaces by $\mathcal{M}$.

The $r-$model spaces are among the most easily described and thoroughly
studied elements of $\mathcal{F}$ \cite{GMR} and so it is interesting to know
how typical they are as elements of $\mathcal{F}$ and what distinguishes them
from the general elements of $\mathcal{F}$. Here we develop several criteria
that characterize the $H\in\mathcal{F}$ that are in $\mathcal{M}$, including
criteria based on the form of the associated set in complex hyperbolic space,
on the values of certain extremal multipliers, on the existence of a
conjugation operator taking the reproducing kernels to their dual basis, and
on having a Gram matrix that is an orthogonal matrix.

The next section has preliminary notation, definitions and results. In Section
3 we state our main results, the proofs are in Section 4. A final section
contains some comments.

\section{Background, Definitions, Notation}

\subsection{RKHS}

For general background about spaces in $\mathcal{F}$ we refer to \cite{AM} and
\cite{ARSW2}.

A finite dimensional RKHS is a finite dimensional Hilbert space $H$ together
with a designated basis $\mathfrak{K}=\mathfrak{K}(H)=\left\{  k_{i}\right\}
_{i=1}^{n}\subset H$ of vectors called reproducing kernels. We denote the dual
basis by $\mathfrak{K}^{\#}=\{k_{i}^{\#}\}_{i=1}^{n};$ that is, $k_{j}^{\#}\in
H$ and $\left\langle k_{i},k_{j}^{\#}\right\rangle =\delta_{ij}$ $1\leq
i,j\leq n.$ Notice that the Hilbert space $H$ with vectors of $\mathfrak{K}%
^{\#}$ as its set of reproducing kernels is also a RKHS, we denote it by
$H^{\#}.$ Let $\mathbf{K}=\mathbf{K}(H)$ be the Gram matrix of $\mathfrak{K,}$
the $n\times n$ matrix $\left(  k_{ij}\right)  $ with $k_{ij}=\left\langle
k_{i},k_{j}\right\rangle $ $1\leq i,j\leq n.$ Similarly let $\mathbf{K}^{\#}$
be the Gram matrix of $\mathfrak{K}^{\#}\mathfrak{.}$ It is not hard to check
that $\mathbf{K}^{\#}=\mathbf{K}^{-1}.$

We will use the metric $\delta=\delta_{H}$ defined on $\mathfrak{K}$ or,
equivalently, the index set of $\mathfrak{K.}$ For each $i$ let $P_{i}$ be the
orthogonal projection onto the span of the kernel function $k_{i}$ and let
$\left\Vert \cdot\right\Vert $ denote the operator norm. We set $\delta
(k_{i},k_{j})=\delta(i,j)=\left\Vert P_{i}-P_{j}\right\Vert .$\ There is also
a formula for $\delta$ in terms of Gram matrix entries;
\begin{equation}
\delta(k_{i},k_{j})=\sqrt{1-\frac{\left\vert k_{ij}\right\vert ^{2}}%
{k_{ii}k_{jj}}}. \label{distance}%
\end{equation}
More information about $\delta$ is in \cite{ARSW1} and \cite{R1}. {}

Two RKHS, $H$ and $\tilde{H},$ with reproducing kernels $\mathfrak{K}=\left\{
k_{i}\right\}  _{i=1}^{n}$ and $\mathfrak{\tilde{K}}=\{\tilde{k}_{i}%
\}_{i=1}^{n}$ are said to be rescalings of each other if there are nonzero
scalars $\left\{  \lambda_{i}\right\}  _{i=1}^{n}$ so that for $1\leq i,j\leq
n$ $\left\langle k_{i},k_{j}\right\rangle =\lambda_{i}\overline{\lambda_{j}%
}\left\langle \tilde{k}_{i},\tilde{k}_{j}\right\rangle .$ We denote this
equivalence relation by $H\sim$ $\tilde{H}$. Equivalently the Gram matrices of
the two spaces are related through conjugation by a diagonal matrix with
nonzero entries in which case we say that that the two matrices are rescalings
of each other. Most of the conditions we consider interact well with this
equivalence relation; for instance if $H\sim$ $\tilde{H}$ then $\delta
_{H}(k_{i},k_{j})=\delta_{\tilde{H}}(\tilde{k}_{i},\tilde{k}_{j})$ and
$H^{\#}\sim$ $(\tilde{H})^{\#}$

\subsection{\textbf{Model Spaces }}

The main facts we use about model spaces are in \cite{GP}; a general reference
is \cite{GMR}.

Suppose $B$ is a Blaschke product with simple zeros at the points $\left\{
z_{i}\right\}  _{i=1}^{n}\subset\mathbb{B}^{1}.$ Thus $B=\Pi_{i}B_{i}$ where,
for $x_{i}\neq0$
\begin{equation}
B_{i}(z)=\frac{\left\vert x_{i}\right\vert }{x_{i}}\frac{x_{i}-z}%
{1-\overline{x_{i}}z} \label{factor}%
\end{equation}
and $B_{i}(z)=z$ if $x_{i}=0.$ The reproducing kernels for the space $K_{B}$
are the functions $\left\{  k_{i}\right\}  \subset K_{B}$ defined by%

\[
k_{i}(w)=\frac{1}{1-\left\langle w,x_{i}\right\rangle }.
\]
To see this recall from Hardy space theory that for each $i$ and each $f\in
H^{2}$ we have $\left\langle f,k_{i}\right\rangle =f(x_{i}).$ With this in
hand it follows that $k_{i}\in\left(  BH^{2}\right)  ^{\perp}=K_{B}.$ These
two facts together insure that the $k_{i}$ are the reproducing kernels. It
then follows with a bit of computation that%
\begin{equation}
\delta(k_{x_{i}},k_{x_{j}})=\delta(x_{i},x_{j})=\left\vert \frac{x_{i}-x_{j}%
}{1-\overline{x_{i}}x_{j}}\right\vert =\left\vert B_{i}(x_{j})\right\vert
=\left\vert B_{j}(x_{i})\right\vert . \label{disk psh}%
\end{equation}

A conjugation operator on a Hilbert space $H$ is an isometric conjugate linear
map $J$ of $H$ to itself that is an involutive automorphism; that is
$\left\Vert Jh\right\Vert =\left\Vert h\right\Vert $ for all $h\in H$ and
$J^{2}=I.$ Each $K_{B}$ carries a conjugation operator $J=J_{B}$ given by%
\begin{equation}
J_{B}f=B\overline{zf}. \label{j}%
\end{equation}
In (\ref{j}) functions in $H^{2}$ are identified with their boundary values
and (\ref{j}) is an equation involving functions on $\mathbb{T}^{1},$ the
boundary circle of $\mathbb{B}^{1}.$ Using (\ref{j}) and the fact that each of
the factors $B_{i}$ of (\ref{factor}) is unimodular on $\mathbb{T}^{1}$ it is
not hard to check that the action of $J_{B}$ is on the kernel functions is
given by%
\[
Jk_{i}(x_{j})=\left\{
\begin{array}
[c]{lll}%
B(x_{j})/(x_{j}-x_{i})\text{ } & \text{if } & i\neq j\\
B^{\prime}(x_{i}) & \text{if } & i=j
\end{array}
\right.  ,
\]
\cite{GP}, \cite{GMR}. Hence for $1\leq i,j\leq n$
\begin{equation}
\left\langle k_{i},Jk_{j}\right\rangle =B^{\prime}(x_{i})\delta_{i,j}.
\label{almost conjugation}%
\end{equation}
Because $B$ has only simple zeros none of the $B^{\prime}(x_{i})$ are zero$.$
Thus it is almost true that $J$ maps the basis of reproducing kernels of
$K_{B}$ to its dual basis. By rescaling we can make that exactly true. We
describe the general pattern.

\begin{proposition}
\label{exact conjugation}Suppose $H$ is a RKHS with kernel functions $\left\{
k_{i}\right\}  _{i=1}^{n}.$ Suppose $J$ is a conjugation operator on $H$ and
that for nonzero $\left\{  c_{i}\right\}  $ we have
\[
\left\langle k_{i},Jk_{j}\right\rangle =c_{i}\delta_{ij}\text{ \ \ }1\leq
i,j\leq n.
\]
Let $\tilde{H}$ be the RKHS obtained by using the same Hilbert space as $H$
and the reproducing kernels $\{\tilde{k}_{i}\}$ defined by $\tilde{k}%
_{i}=c_{i}^{-1/2}k_{i},$ $1\leq i\leq n.$ Then $\tilde{H}$ $\sim$ $H$ and $J$
is a conjugation operator on $\tilde{H}$ which maps the reproducing kernels of
$\tilde{H}$ to their dual basis; that is
\[
\left\langle \tilde{k}_{i},J\tilde{k}_{j}\right\rangle =\delta_{ij}\text{
\ \ }1\leq i,j\leq n.
\]

\begin{proof}
The first two statements are clear. To check the third we compute, for $1\leq
i,j\leq n$%
\begin{align*}
\left\langle \tilde{k}_{i},J\tilde{k}_{j}\right\rangle  &  =\left\langle
c_{i}^{-1/2}k_{i},Jc_{j}^{-1/2}k_{j}\right\rangle =\left\langle c_{i}%
^{-1/2}k_{i},\overline{c_{j\text{ }}}^{-1/2}Jk_{j}\right\rangle \\
&  =c_{i}^{-1/2}c_{j}^{-1/2}\left\langle k_{i},Jk_{j}\right\rangle
=c_{i}^{-1/2}c_{j}^{-1/2}c_{i}\delta_{ij}=\delta_{ij}.
\end{align*}

\end{proof}
\end{proposition}

We will call such an $\tilde{H}$ orthogonal and say its rescaling $H$ is
$r-$orthogonal. The name is based on the following simple proposition. Recall
that a matrix is called orthogonal if its inverse equals its transpose.

\begin{proposition}
[{\cite[Prop. 8]{R1}}]\label{conjugation = orthogonal}$H,$ a finite
dimensional RKHS, is orthogonal if and only if its Gram matrix is an
orthogonal matrix.
\end{proposition}

Suppose $H$ is a finite dimensional RKHS with kernel functions $\left\{
k_{i}\right\}  $ and Gram matrix $\mathbf{K}=\left(  k_{ij}\right)  $. Let
$\{k_{i}^{\#}\}$ be the dual basis and hence the set of kernel functions of
$H^{\#}$ and write its Gram matrix as $\mathbf{K}^{\#}=(k_{ij}^{\#}).$ It is
always true that the transpose of $\mathbf{K,}$ $\mathbf{K}^{t}$ is equal to
the entrywise conjugate matrix $\mathbf{\bar{K}}$. It is also always true that
$\mathbf{K}^{\#}=\mathbf{K}^{-1}.$ Thus $H$ is orthogonal exactly if
$\mathbf{KK}^{t}=\mathbf{K\bar{K}}=\mathbf{KK}^{\#}=I.$ Formulated in terms of
matrix entries this is, for $1\leq i,j\leq n,$
\begin{equation}
\sum_{s=1}^{n}k_{is}k_{js}=\sum_{s=1}^{n}k_{is}\overline{k_{sj}}=\sum
_{s=1}^{n}k_{is}k_{sj}^{\#}=\delta_{ij}. \label{orthogonal matrix}%
\end{equation}

\begin{corollary}
\label{is also}If $H\in\mathcal{M}$ then $H^{\#}\in\mathcal{M}$. Specifically,
if $H$ $\sim$ $K_{B}$ where $B$ is the Blaschke product with zeros $\left\{
x_{i}\right\}  $ then $H^{\#}$ $\sim$ $K_{B^{\#}}$ where $B^{\#}$ is the
Blaschke product with zeros $\left\{  \overline{x_{i}}\right\}  .$
\end{corollary}

\begin{proof}
The Gram matrix of $H^{\#}$ is the inverse matrix of the Gram matrix of $H$
and hence is a rescaling of the inverse of $\mathbf{K}_{B}$, the Gram matrix
of $K_{B}.$ From the previous proposition $K_{B}$ is $r-$orthogonal hence
$\mathbf{K}_{B}$ is a rescaling of an orthogonal matrix and thus $\left(
\mathbf{K}_{B}\right)  ^{-1}$ is a rescaling of $\overline{\mathbf{K}_{B}}$,
the matrix of complex conjugates of entries of $\mathbf{K}_{B}$. It is a
direct consequence of the definitions that $\overline{\mathbf{K}_{B}}$
$=\mathbf{K}_{B^{\#}}$. In sum, the Gram matrix of $H^{\#}$ is a rescaling of
the Gram matrix of $K_{B^{\#}}.$ Hence $H^{\#}$ $\sim$ $K_{B^{\#}}.$
\end{proof}

This corollary shows that the Gram matrix of $K_{B}$ is a rescalings of the
complex conjugate of the Gram matrix of $K_{B}^{\#}.$ A similar result holds
for some infinite Blaschke products $B$ and that fact has been used in the
study of $H^{2}$ interpolating sequences \cite[Exercise 9.54]{AM} \cite[Sec.
7.3.1]{ARSW2}.

Because $H^{\#\#}=H$ the converse of the corollary also holds and thus we have
an necessary and sufficient condition for $H^{\#}\in\mathcal{M}$. On the other
hand it is not clear what conditions insure $H^{\#}\in\mathcal{F}$. We discuss
that question briefly in Section\ \ref{C}.

\subsection{Multiplier Algebras}

If $H\in\mathcal{F}$ then there is a particularly close relationship between
$H$ and $M(H).$ In particular suppose $X$ is the index set of $\mathfrak{K}%
(H),$ $Y\subset X$ and $x\in X\smallsetminus Y.$ Let $m(Y,x)\in M(H)$ be the
multiplier of norm one which vanishes on $Y$ and maximizes $\operatorname{Re}%
m(Y,x).$ Let $h(Y,x)$ be the vector in $H$ of norm one which on $Y$ and
maximizes $\operatorname{Re}h(Y,x).$

\begin{proposition}
[{\cite[Prop. 6.27]{ARSW2}}]\label{extreme}In the situation just described and
with $k_{x}$ the reproducing kernel for $x$
\begin{equation}
m(Y,x)\frac{k_{x}}{\left\Vert k_{x}\right\Vert }=h(Y,x). \label{extremal}%
\end{equation}

\end{proposition}

With this as a starting point it is not hard to show that for $H\in
\mathcal{F}$ the metric $\delta$ on $X,$ the index set of $\mathfrak{K}(H),$
is the same as the Gleason metric on $X$ induced by $M(H),$ that is%
\begin{equation}
\delta(x,y)=\max\left\{  \operatorname{Re}m(x):m\in M(H),m(y)=0,\left\Vert
m\right\Vert =1\right\}  , \label{gleason}%
\end{equation}
\cite[Remrk 7.2]{ARSW2}. The equality of the term on the left (\ref{disk psh})
and the last two terms on the right is an instance of this general fact.

\subsection{Drury Arveson Spaces}

For $m=1,2,....$ the Drury-Arveson space $DA_{m}$ is the Hilbert space of
functions defined on the complex $m-$ball $\mathbb{B}^{m}$ which is the
closure of the span of the reproducing kernels $\left\{  k_{z}(w)=\frac
{1}{1-\left\langle w,z\right\rangle }:z\in\mathbb{B}^{m}\right\}  .$ Here
$\left\langle w,z\right\rangle $ is the standard Hermitian inner product on
$\mathbb{C}^{n}.$ References for this space and its properties include
\cite{AM}, \cite{ARSW2}, and \cite{Sh}.

For $X$ a finite subset of some $\mathbb{B}^{m}$ we define the associated
space $DA_{m}(X)$ to be the subspace of $DA_{m}$ spanned by $\left\{
k_{x}:x\in X\right\}  $ and having those functions as kernel functions. In
particular if $m=1$ then the kernel functions of $DA_{1}(X)$ are the same as
those of the model space $K_{B_{X}}$ for $B_{X}$ the Blaschke product with
zeros at the points of $X;$ thus those two spaces are the same (or,
pedantically, are trivial rescalings of each other). When $m>1$ then the
details of our discussion are effectively independent of $m$ and we will not
keep track of that index; this is discussed in \cite{R1}.

\begin{definition}
\label{pick}We say $H\in\mathcal{F}$ if there is an integer $m$ and finite
$X(H)\subset\mathbb{B}^{m}$ such that $H\sim$ $DA_{m}(X(H)).$
\end{definition}

Although this definition suits our purposes it is very different from the
traditional definition of finite dimensional RKHS with complete Pick kernels.
The equivalence of our definition with the traditional one is a basic theorem
in the subject \cite{AM}.

\subsection{Hyperbolic Geometry}

The unit ball $\mathbb{B}^{m}\subset\mathbb{C}^{m}$ is a model for complex
hyperbolic $m-$space $\mathbb{CH}^{m}$ . This is discussed in detail in
\cite{Go}; here we will just recall a few pieces of information that we need.

The space $\mathbb{CH}^{m}$ carries a transitive set of orientation preserving
automorphisms. In the ball model these are realized by the group of conformal
automorphisms of $\mathbb{B}^{m}.$

The space $\mathbb{CH}^{m}$ carries several natural metrics, of particular
interest to us is the pseudohyperbolic metric $\Delta$ which can be defined by
setting, for $z\in\mathbb{B}^{m},$ $\Delta(0,z)=\left\vert z\right\vert $ and
requiring $\Delta$ to be invariant under the automorphism group. (We note for
context, but will not use the fact, that the length metric generated by
$\Delta$ is the classical Bergman metric on the ball.) In particular, if
$z,w\in\mathbb{B}^{1}$ then%
\begin{equation}
\Delta(z,w)=\left\vert \frac{z-w}{1-\bar{z}w}\right\vert . \label{delta}%
\end{equation}

The disk $\mathbb{B}^{1}$ sits inside $\mathbb{B}^{m}=\mathbb{CH}^{m}$ as a
totally geodesically embedded manifold of complex dimension one, called a
complex geodesic. All the other totally geodesically embedded complex one
manifolds, the other complex geodesics, are the images of that unit disk under
the group of conformal automorphisms of the ball. In particular the group of
automorphisms acts transitively on the set of complex geodesics.

\subsection{Hyperbolic Geometry and $\mathcal{F}$}

The structure of the spaces $DA_{m}(X)$ is closely related to the geometry of
$X$ regarded as a subset of $\mathbb{CH}^{n}.$ This theme is developed in
\cite{ARSW1}, \cite{R1}, and \cite{R2}.

Comparing (\ref{delta}) with (\ref{disk psh}) we see that if $B$ is a Blaschke
product which vanishes at $z$ and $w$ and possibly other points, and with
$\delta$ denoting the metric on $\mathfrak{K}(B)$ then we have $\Delta
(z,w)=\delta(z,w).$ In fact this is the general pattern.{}

\begin{proposition}
[{\cite[4.1]{R1}}]Suppose $X$ is a finite set in some $\mathbb{B}^{m}$ with
$z,w\in X.$ Let $\delta$ be the metric on $X$ which is the index set of
$\mathfrak{K}(DA_{m}(X))$ and let $\Delta$ be the pseudohyperbolic metric on
$\mathbb{CH}^{m}$ restricted to $X,$ then $\Delta(z,w)=\delta(z,w).$
\end{proposition}

There is another fundamental relation between elements of $\mathcal{F}$ and
hyperbolic geometry. Given $H\in\mathcal{F}$ we know from Definition
\ref{pick} that $H\sim$ $DA_{m}(X)$ for some finite $X.$ In fact we can
describe the possible choices for $X.$

\begin{theorem}
[{\cite[Thm 7]{R1}}]\label{R1}Suppose $X,Y$ are finite sets in some
$\mathbb{CH}^{m}.$ $DA(X)\sim$ $DA(Y)$ if and only if there is an automorphism
$\Phi$ of $\mathbb{CH}^{m}$ with $\Phi(X)=Y.$
\end{theorem}

In sum, questions of equivalence of elements of $\mathcal{F}$ under rescaling
are equivalent to questions about congruence of finite sets in complex
hyperbolic space.

\subsection{Subspaces\label{subspaces}}

Suppose $H$ is a finite dimensional RKHS with kernel functions $\left\{
k_{\alpha}\right\}  _{\alpha\in X}.$ We say $J$ is a regular subspace of $H$
if there is a $Y\subset X$ and $J$ is the Hilbert space spanned by $\left\{
k_{\beta}\right\}  _{\beta\in Y}$ and is regarded as a RKHS with the $\left\{
k_{\beta}\right\}  _{\beta\in Y}$ as its its kernel functions. It follows from
the definitions that if $H\in\mathcal{F}$ then $J\in\mathcal{F}$, and
similarly for membership in $\mathcal{M}$.

Statements which can be formulated using Gram matrix entries can be passed
between a space $H$ and a regular subspace $J$. For instance, given $H,J,X,Y$
as above and $y,y^{\prime}\in Y$ then the value $\delta(k_{y},k_{y^{\prime}})$
does not depend on whether we regard the kernel functions as in $J$ or in $H.$
Another instance of this, one we use in the proof below in showing Statement
(2) implies Statement (1), is that if $J$ is a regular subspace of a $DA(X),$
and hence is $DA(Y)$ for some $Y\subset X$ then knowing that three points of
$Y$ are in a complex geodesic implies that the "same" points, regarded as
inside of $X,$ and hence as index points for kernel functions in $DA(X)$, are
also in a complex geodesic.

Another fact about regular subspaces which we will use is that if
$H\in\mathcal{F}$ and $J$ is a regular subspace of $H$ then it is a
consequence of the Pick property of $H$ than given a multiplier $m_{J}\in
M(J)$ there is an extension to a multiplier $m_{H}\in M(H)$, defined on all of
$H,$ with the same norm, $\left\Vert m_{J}\right\Vert =\left\Vert
m_{h}\right\Vert .$

At times we will use these facts without mention.

\section{The Results}

Our main result is the following

\begin{theorem}
\label{A} Suppose $H\in\mathcal{F}$, that is for some $m,$ and $X=\left\{
x_{i}\right\}  _{i=1}^{n}\subset\mathbb{B}^{m},$ we have $H\sim$ $DA_{m}(X)$.
The following are equivalent:

\begin{enumerate}
\item $X$ lies in a single complex geodesic in $\mathbb{B}^{m}=\mathbb{CH}%
^{m}.$

\item For each $1<i<j\leq n$, $\left\{  x_{1},x_{i},x_{j}\right\}  $ lies in a
single complex geodesic in $\mathbb{B}^{m}=\mathbb{CH}^{m}.$

\item There is a renumbering of $X$ after which, with $m_{1}\in M(H)$ the
multiplier of norm one which satisfies $m_{1}(x_{j})=0,$ $j=2,...,n$ and which
maximizes $\operatorname{Re}m_{1}(x_{1})\ $we have
\begin{equation}
m_{1}(x_{1})=\prod_{j=2}^{n}\delta(x_{1},x_{j}). \label{multiplier = product}%
\end{equation}

\item $H$ is $r-$orthogonal. That is $H\sim$ $\tilde{H}$ for some $\tilde{H}$
which carries a conjugation operator taking the basis of kernels of $H$ to the
dual basis.

\item $H\sim$ $\tilde{H}$ \ for an $\tilde{H}$ which has a Gram matrix which
is an orthogonal matrix.

\item $H$ is an $r-$model space; that is there is a finite Blaschke product
with simple zeros $B$ such that $H\sim$ $K_{B}.$
\end{enumerate}
\end{theorem}

\section{The Proofs{}}

Some parts of the proof follow from our earlier discussion and we begin with
those. First we show (1) and (6) are equivalent. Suppose (6) holds and let
$X(B)$ be the zero set of the Blaschke product $B.$ As we noted earlier the
spaces $K_{B}$ and $DA_{1}(X(B))$ have the same kernel functions and hence are
the same space. Furthermore $DA_{1}(X(B))$ is in the form described in (1),
that is $X(B)$ is in the unit disk which is a complex geodesic, both on its
own and as a subset of any larger $\mathbb{B}^{m}=\mathbb{CH}^{m}.$ On the
other hand, if (1) holds then we can use the fact that the group of
automorphisms acts transitively on the set of complex geodesics to select an
automorphism $\Phi$ mapping the geodesic containing $X$ to the unit disk. By
Theorem \ref{R1} $DA_{m}(X)\sim$ $DA_{1}(\Phi(X))$ and as before $DA_{1}%
(\Phi(X))\sim K_{B}$ where $B$ is now the Blaschke product with zeros at
$\Phi(X).$ Thus we have (6) as required.

That (4) and (5) are equivalent is Proposition \ref{conjugation = orthogonal}.

It is immediate that (1) implies (2). To see that (2) implies (1) recall that
any two points in $\mathbb{CH}^{n}$ are contained in a unique complex
geodesic. Hence, by the same argument we would use to study colinear points in
Euclidian space we see that if (2) holds then so does (1).

The demonstrations that (6) implies (4) and that (6) implies (3) both use the
function theory of the Hardy space applied to $K_{B}\subset H^{2}.$ That (6)
implies (4) follows from Proposition \ref{exact conjugation} and the
discussion proceeding it. To see that (6) implies (3) note that from classical
function theory on the Hardy space, in particular Pick's theorem, the extreme
value $m_{1}(x_{1})$ for the multiplier $m_{1}\in M(K_{B})$ is the same as the
extreme value for the $H^{\infty}$ interpolation problem of finding.%
\[
\max\left\{  \operatorname{Re}g(x_{1}):g(x_{2})=...=g(x_{n})=0,\text{ }%
\sup_{z\in\mathbb{B}^{1}}\left\vert g(z)\right\vert \leq1\right\}  .
\]
That problem is solved using Blaschke products and we find that the maximum is
$\left\vert D(x_{1})\right\vert $ where $D$ is a Blaschke product with zeros
$\left\{  x_{2},...,x_{n}\right\}  .$ The relation between Blaschke factors
and $\delta$ given in (\ref{disk psh}) completes the argument.

It remains to show that (3) implies (6) and that (4) implies (6). If $H$ is
three dimensional then both implications are known. We recall those results
and then use them to pass to the general cases.

\begin{proposition}
\label{3 implies 6}Suppose $H\in\mathcal{F}$ is three dimensional with kernel
functions $\left\{  k_{a},k_{b},k_{c}\right\}  .$ Let $s\in M(H)$ the
multiplier of norm one which vanishes at $k_{b}$ and $k_{c}$ and maximizes
$\operatorname{Re}s(k_{a}).$ If
\begin{equation}
s(k_{a})=\delta(k_{a},k_{b})\delta(k_{a},k_{c}). \label{goal}%
\end{equation}
then $H\in\mathcal{M}$.
\end{proposition}

This is Proposition 27 of \cite{R1}. The proof is based on a detailed analysis
of $DA(X)$ and its multipliers for three point sets $X.$

The next result is from \cite{R1} where it is stated with an oversight. We
will say a three dimensional RKHS is non-degenerate if no entry of its Gram
matrix is zero, equivalently no two kernel functions are orthogonal. The
property is preserved under rescaling and under passage to regular subspaces.

\begin{proposition}
\label{4 implies 6}Suppose $H$ is a non-degenerate three dimensional RKHS
which is $r-$orthogonal, then $H\in\mathcal{M}$.
\end{proposition}

If $H\in\mathcal{F}$ then $H$ is automatically non-degenerate. However the
proposition does not have the assumption that $H\in\mathcal{F}$.

This is Theorem 28 of \cite{R1} where the requirement of non-degeneracy was
omitted. The shape of the proof is that finding rescaling parameters that will
transform the Gram matrix of $H$ into an orthogonal matrix involves solving a
system of equations. For those equations to have a solution a determinant
involving functions of the Gram matrix entries must vanish. That vanishing
gives an equation for the Gram matrix entries which is equivalent to knowing
$H\sim DA(X)$ for $X$ in a single complex geodesic. (This is the place where
the proof in \cite{R1} fails. Without the non-degeneracy assumption the
equation involving the Gram matrix entries may trivialize.)

Suppose now that (3) holds and we want to establish (6). After rescaling we
may suppose $H=DA(X_{H})$ for some $X_{H}\subset\mathbb{B}^{m}.$ Let $J$ be
the regular subspace of $H$ with kernel functions $\left\{  k_{a},k_{b}%
,k_{c}\right\}  .$ If we have (\ref{goal}) then by Proposition
\ref{3 implies 6} $J\in\mathcal{M}$ and thus $J\sim DA(X_{J})$ for some
$X_{J}\subset\mathbb{B}^{1}$; in particular $X_{J}$ is in a single complex
geodesic. As discussed in Section \ref{subspaces} this implies that the points
of $X$ corresponding to those three kernel functions also lie in a single
geodesic. The choice of which three kernel functions we considered was
arbitrary and hence we would have (2) for $H.$ We have already seen that (2)
implies (1) which implies (6).

Thus we are done if we can show that for any choice of three kernel functions
(\ref{goal}) holds. From (\ref{gleason}) we know that given any kernel
functions $k_{\alpha},k_{\beta}$ in an $H\in\mathcal{F}$ there will be a
multiplier $m_{\alpha\beta}\in M(H)$ of norm one which satisfies
\begin{equation}
m_{\alpha\beta}(k_{\alpha})=0,\text{ \ \ }m_{\alpha\beta}(k_{\beta}%
)=\delta(k_{\alpha},k_{\beta}). \label{gleason extremal}%
\end{equation}
We now proceed by contradiction. Suppose we have found three kernel functions,
$\left\{  k_{i}\right\}  _{i=1,2,3},$ for which (\ref{goal}) fails and denote
their span by $J.$ Using the notation of (\ref{gleason extremal}) form the
multiplier $r=m_{12}m_{13}\in M(J).$ This multiplier is a candidate for the
extremal problem defining the multiplier $s$ in (\ref{goal}) and by evaluating
if we see that the left hand side in (\ref{goal}) will never be smaller than
the right. Hence by our construction of $J$ (\ref{goal}) fails because the
left hand side is larger. Thus we have a multiplier $t\in M(J)$ of norm one
with
\begin{equation}
t(k_{1})>\delta(k_{1},k_{2})\delta(k_{1},k_{3}). \label{inequality}%
\end{equation}
As noted in Section \ref{subspaces} the multipliers $m_{\alpha\beta}$ and $t$
all have norm preserving extensions to elements in $M(H).$ We will regard
those extensions as having been made and use the same notation for the
extended multipliers. Staying with the notation in (\ref{gleason extremal})
consider the multiplier in $q\in M(H)$ given by%
\[
q=t\prod_{j=4}^{n}m_{1j}.
\]
By the Banach algebra property of $M(H)$ $q$ has norm at most one and by
construction it vanishes at $k_{2},...,k_{n};$ hence $q$ is a competitor in
the extremal problem defining $m_{1}.$ Furthermore, comparing the definition
of $q$, (\ref{gleason extremal}), and (\ref{inequality}) we see that
$q(k_{1})$ is larger than the right hand side of (\ref{multiplier = product}).
This inequality contradicts the extremal value suggested by
(\ref{multiplier = product}) and hence contradicts the assumption that (3)
holds. This completes the proof that (3) implies (6)

We now show (4) implies (6). We are given $H\in\mathcal{F}$ that is $n$
dimensional and is $r-$orthogonal. Let $\left\{  k_{i}\right\}  _{i=1}^{n}$ be
an arbitrary numbering of the kernel functions of $H$ and let $H_{-}%
\in\mathcal{F}$ be the $n-1$ dimensional regular subspace of $H$ spanned by
the kernel functions $\left\{  k_{i}\right\}  _{i=1}^{n-1}.$ We will show that
$H_{-}$ is $r-$orthogonal. Repeating this shows that for any three kernel
functions $\left\{  k_{\alpha},k_{\beta},k_{\gamma}\right\}  $ their span,
$H_{\alpha\beta\gamma},$ is $r-$orthogonal. By Proposition \ref{4 implies 6}
this insures $H_{\alpha\beta\gamma}\in M.$ Hence after rescaling
$H_{\alpha\beta\gamma}=DA(X_{\alpha\beta\gamma})$ for a three point set
$X_{\alpha\beta\gamma}$ in the unit disk. As we noted in Section
\ref{subspaces} this insures that when $H$ is rescaled as $DA(X)$ for some $X$
in some $\mathbb{B}^{m}$ then the points of $X$ corresponding to the kernel
functions $\left\{  k_{\alpha},k_{\beta},k_{\gamma}\right\}  $ will lie in a
single complex geodesic in $\mathbb{B}^{m}.$ The numbering of the kernel
functions of $H$ was arbitrary and hence this establishes (2) for $H$ from
which (1) follows and then (6).

To prove the reduction we start with $n-$dimensional $H\in\mathcal{F}$ with
kernel functions $\left\{  k_{j}\right\}  _{j=1}^{n}$ which, by rescaling, we
can assume has an orthogonal Gram matrix $\mathbf{K}=(k_{ij})_{i,j=1}^{n}$. We
denote the dual basis, the kernel functions of $H^{\#}$, by $\{k_{i}%
^{\#}\}_{i=1}^{n}$ and write its Gram matrix as $\mathbf{K}^{\#}=(k_{ij}%
^{\#})_{ij=1}^{n}.$ Because $H$ is orthogonal we have $\mathbf{K}%
^{-1}=\mathbf{K}^{\#}=\mathbf{K}^{t}$ and hence the condition that $H,$ and
thus $\mathbf{K}$, is orthogonal is expressed by the equations, for $1\leq
i,j\leq n.$%
\begin{equation}
\sum_{s=1}^{n}k_{is}k_{sj}^{\#}=\delta_{ij}. \label{use}%
\end{equation}
Let $H_{-}$ be the regular subspace of $H$ spanned by the kernel functions
$\left\{  k_{i}\right\}  _{i=1}^{n-1}.$ The inner product of those $k_{j}$ is
the same whether they are regarded as vectors in $H$ or $H_{-}$ and hence the
Gram matrix of $H_{-}$ is the upper left $\left(  n-1\right)  \times\left(
n-1\right)  $ block of the Gram matrix of $H:\mathbf{K}_{-}=\left(
k_{ij}\right)  _{ij=1}^{n-1}.$ Let $\{k_{-i}^{\#}\}$ be the dual basis of the
basis $\left\{  k_{i}\right\}  _{i=1}^{n-1}$ and let $\mathbf{K}_{-}%
^{\#}=\left(  k_{ij}^{\#}\right)  _{ij=1}^{n-1}$ be its Gram matrix. For ease
of reading we set $g_{j}=k_{-j}^{\#}$ and $g_{ij}=k_{-ij}^{\#}.$

The vectors $\{k_{j}^{\#}\}_{j=1}^{n-1}$ are nearly but not quite the
$\left\{  g_{j}\right\}  _{j=1}^{n-1}.$ They give the correct inner products;
for $1\leq i,j\leq n-1$%
\begin{equation}
\left\langle k_{i},k_{j}^{\#}\right\rangle =\delta_{ij}, \label{want}%
\end{equation}
but they are not in $H.$ To obtain vectors in $H$ which satisfy the analog of
(\ref{want}) we apply $P,$ the orthogonal projection from $H$ to $H_{-}.$ That gives%

\[
\left\langle k_{i},P(k_{j}^{\#})\right\rangle =\delta_{ij};
\]
and hence the $P(k_{j}^{\#})$ are our desired $g_{j},$ $j=1,...,n-1.$

We know $P(k_{j}^{\#})=k_{j}^{\#}-Q(k_{j}^{\#})$ where $Q$ is the projection
complimentary to $P$ and we use that to compute $g_{j}.$ $Q$ is the projection
onto the orthocomplement of $H_{-}$ in $H$ and by the definition of the dual
basis that is the subspace of $H$ spanned by $k_{n}^{\#}$. Hence for any $b$
\[
P(b)=b-\left\langle b,\frac{k_{n}^{\#}}{\left\Vert k_{n}^{\#}\right\Vert
}\right\rangle \frac{k_{n}^{\#}}{\left\Vert k_{n}^{\#}\right\Vert }.
\]
In particular, for $1\leq j\leq n-1$%
\begin{align*}
g_{j}  &  =k_{j}^{\#}-\left\langle k_{j}^{\#},\frac{k_{n}^{\#}}{\left\Vert
k_{n}^{\#}\right\Vert }\right\rangle \frac{k_{n}^{\#}}{\left\Vert k_{n}%
^{\#}\right\Vert }\\
&  =k_{j}^{\#}-\frac{k_{jn}^{\#}}{k_{nn}^{\#}}k_{n}^{\#}.
\end{align*}
We want to evaluate $\mathbf{K}_{-}^{\#}=\left(  g_{ij}\right)  ;$%

\begin{align*}
g_{ij}  &  =\left\langle g_{i},g_{j}\right\rangle =\left\langle k_{i}%
^{\#}-\frac{k_{in}^{\#}}{k_{nn}^{\#}}k_{n}^{\#},\text{ }k_{j}^{\#}%
-\frac{k_{jn}^{\#}}{k_{nn}^{\#}}k_{n}^{\#}\right\rangle \\
&  =\left\langle k_{i}^{\#},\text{ }k_{j}^{\#}\right\rangle -\left\langle
\frac{k_{in}^{\#}}{k_{nn}^{\#}}k_{n}^{\#},k_{j}^{\#}\right\rangle +0\\
&  =k_{ij}^{\#}-\frac{k_{in}^{\#}k_{nj}^{\#}}{k_{nn}^{\#}}%
\end{align*}
To show that $\mathbf{K}_{-}$ is an orthogonal matrix we will show
\[
\sum_{s=1}^{n-1}k_{-is}k_{-sj}^{\#}=\sum_{s=1}^{n-1}k_{is}g_{sj}=\delta_{ij}.
\]
We compute%
\begin{align*}
\sum_{s=1}^{n-1}k_{is}g_{sj}  &  =\sum_{s=1}^{n-1}k_{is}\left(  k_{sj}%
^{\#}-\frac{k_{sn}^{\#}k_{nj}^{\#}}{k_{nn}^{\#}}\right) \\
&  =\sum_{s=1}^{n-1}k_{is}k_{sj}^{\#}-\sum_{s=1}^{n-1}\frac{k_{is}k_{sn}%
^{\#}k_{nj}^{\#}}{k_{nn}^{\#}}\\
&  =\left(  \delta_{ij}-k_{in}k_{nj}^{\#}\right)  +\frac{\left(  -\delta
_{in}+k_{in}k_{nn}^{\#}\right)  k_{nj}^{\#}}{k_{nn}^{\#}}\\
&  =\delta_{ij}.
\end{align*}
In the passage from the second line to the third we used (\ref{use}) for the
index pair $\left(  i,j\right)  $ and for the index pair $\left(  i,n\right)
.$ The passage to the final line used the fact that $i<n$ and hence
$\delta_{in}=0.$

\section{Comments and Variations}

\subsection{Reformulations of $m_{1}(x_{1})=\prod_{j=2}^{n}\delta(x_{1}%
,x_{j})$}

Suppose $m_{1}\in M(H)$ is the multiplier described in (3) of Theorem \ref{A}.
Taking note of Proposition \ref{extreme} we can write (\ref{extremal}) for
this multiplier and obtain%
\[
m_{1}\frac{k_{1}}{\left\Vert k_{1}\right\Vert }=h.
\]
Here $h$ is the function of norm one which vanishes at $x_{2},...,x_{n}$ and
maximizes $\operatorname{Re}h(x_{1}).$ The space of competitors for $h$ is one
dimensional and hence $h=k_{1}^{\#}/\left\Vert k_{1}^{\#}\right\Vert .$ Using
this and taking the inner product of both sides of the previous equation with
$k_{1}$ we find%
\[
m_{1}(k_{1})\frac{k_{11}}{\left\Vert k_{1}\right\Vert }=\frac{\left\langle
k_{1}^{\#},k_{1}\right\rangle }{\left\Vert k_{1}^{\#}\right\Vert }%
\]
which simplifies as
\[
m_{1}(k_{1})=(\left\Vert k_{1}\right\Vert \left\Vert k_{1}^{\#}\right\Vert
)^{-1}.
\]

We obtain another expression for $m_{1}(k_{1})$ if we consider the idempotent
multiplier in $M(H)$ which takes the value $1$ at $k_{1}$ and is zero
elsewhere. Denote that multiplier by $\operatorname{Idem}_{1}.$ It is in the
one dimensional space spanned by $m_{1}$ and hence $\operatorname{Idem}%
_{1}=m_{1}/m_{1}(k_{1}).$ Thus%
\[
m_{1}(k_{1})=(\left\Vert \operatorname{Idem}_{1}\right\Vert )^{-1}.
\]

Either of these evaluations of $m_{1}(k_{1})$ could be used on the left hand
side of (\ref{multiplier = product}). We had already noted that the $\delta$'s
on the right hand side can be evaluated using the Hilbert space structure $H$
or the using the structure of $M(H).$ Hence we can write versions of
(\ref{multiplier = product}) based entirely on data from $H$ or entirely on
using data from $M(H).$

\subsection{When is $H^{\#}\in\mathcal{F}?$\label{C}}

We noted that Corollary \ref{is also} leads to the fact that $H\in\mathcal{M}$
is a necessary and sufficient condition for $H^{\#}\in\mathcal{M}$. On the
other hand it is not clear what conditions on a RKHS $H,$ even one in
$\mathcal{F}$, will insure $H^{\#}\in\mathcal{F}$. Here we indicate that the
condition $H\in\mathcal{M}$, which is sufficient because $H^{\#}%
\in\mathcal{M\subset F}$, is not necessary, and the condition $H\in
\mathcal{F}$ is not sufficient.

Suppose $X\subset\mathbb{B}^{1}\subset\mathbb{B}^{2}$ and $H=DA(X)$.
Arbitrarily slight modification of $X$ can produce $\tilde{X}\subset
\mathbb{B}^{2}$ which are not contained in a single complex geodesic. For
those $\tilde{X}$ by Theorem \ref{A} $\tilde{H}=DA(\tilde{X})\in$
$\mathcal{F\smallsetminus M}$. If $\tilde{X}$ is sufficiently close to $X$
then the Gram matrix of $\tilde{H}$ is arbitrarily close to the Gram matrix of
$H.$ Furthermore, passing to inverse matrices, the Gram matrix of $\tilde
{H}^{\#}$ is arbitrarily close to the Gram matrix of $H^{\#}.$ We know
$H^{\#}\in\mathcal{M\subset F}$ and we know that being in $\mathcal{F}$ is an
open condition on the Gram matrix of a HSRK \cite{AM} (this is in contrast to
the condition for being in $\mathcal{M}$) hence, if our perturbation was
sufficiently small then $H^{\#}\in\mathcal{F}$.

On the other hand $H\in\mathcal{F}$ is not itself sufficient to insure
$H^{\#}\in\mathcal{F}$. Consider, for $0<a,b<1,$ $X=\left\{  \left(
0,0\right)  ,\left(  a,0\right)  ,\left(  0,b\right)  \right\}  =\mathbb{B}%
^{2},$ in some sense the extreme opposite of $X$ being in a single geodesic.
The Gram matrix of $H=DA(X)$ has the form
\[
G=%
\begin{pmatrix}
1 & 1 & 1\\
1 & \ast & 1\\
1 & \ast & \ast
\end{pmatrix}
\]
The Gram matrix of $H^{\#}$ is $G^{-1}$ and by explicit computation the
$(2,3)$ entry of $G^{-1}$ is $0$ which is impossible for a RKHS in
$\mathcal{F}$.

\subsection{The Use of Hardy Space Theory and the Pick Condition}

In our proof of Theorem \ref{A} we used the function theory of the Hardy space
to study $K_{B}\subset H^{2},$ in particular to prove that (6) implies (3) and
that (6) implies (4). It would be interesting to have a proof of either of
these implication, or of the equivalence of (3) and (4) without involving (6),
that was inside the the theory of RKHS and did not use function theory. In
this context we note the work of Cole, Lewis, and Wermer in \cite{CLW}. They
study conditions on $M(H),$ the multiplier algebra of a RKHS $H,$ which
suffice to insure that $H\in\mathcal{M}$. They have two approaches, one using
operator theory and a second which uses their hypotheses on the multiplier
algebras to, in effect, reconstruct and thus reintroduce parts of Hardy space
function theory.

Many of the steps in the proof of Theorem \ref{A} did not require the Pick
property and hence can be used for any finite dimensional RKHS. However there
were places where we did use the Pick property and do not know the extent to
which it could be avoided. In particular we used Proposition \ref{extreme} to
connect multiplier algebra statements to Hilbert space statements. This fact,
which is a characteristic property of spaces in $\mathcal{F}$, was used in
showing $\delta$ satisfies (\ref{gleason}) and in our analysis of condition
(3) in Theorem \ref{A}.

Also, the hypothesis $H\in\mathcal{F}$ was used in passing from condition (2)
to condition (1). For instance it is not clear that given a general four
dimensional RKHS $H,$ and knowing that all of its regular three dimensional
subspaces are in $\mathcal{M}$, is enough to insure $H\in\mathcal{F}$. If we
could show that, that $H\in\mathcal{F}$, then by the implication (2) implies
(1) in Theorem \ref{A} we would also know $H\in\mathcal{M}$. In this context
it is interesting to note that there are examples due to Quiggen in which all
the regular subspaces of a four dimensional $H$ are in $\mathcal{F}$ but $H$
is not in $\mathcal{F}$. Those are discussed in \cite{R2}.

To see this issue in context, suppose we wanted to show that an orthogonal
$H$, not necessarily in $\mathcal{F}$, satisfied Statement (1), and hence
Statement (6), in Theorem \ref{A}. The reduction to the case of three
dimensional regular subspaces is a linear algebra argument and so continues to
hold. If we assume $H$ is non-degenerate we can then use Proposition
\ref{4 implies 6} to obtain a version of Statement (2). However without
knowing $H\in\mathcal{F}$ it is not clear how to proceed from Statement (2) to
Statement (1).

\subsection{Replacing RKHS by Point Sets in Projective Space}

It is possible to recast much of the previous discussion in the language of
point sets in complex projective space. The viewpoint is intriguing but it is
not clear where it leads. We will be brief and informal.

Suppose we start with the finite dimensional Hilbert space $\mathbb{C}^{n}$
and regard it as a RKHS, $H,$ by selecting a set of basis vectors $\left\{
k_{i}\right\}  $ and declaring those vectors to be reproducing kernels. In
fact any finite dimensional RKHS is a rescaling of such an $H.$

Let $\mathbb{PC}^{n-1}$ be the complex projective space of lines in
$\mathbb{C}^{n};$ for each nonzero vector $v\in\mathbb{C}^{n}$ we denote the
line containing $v$ by $\left[  v\right]  ;$ thus $\left[  v\right]
\in\mathbb{PC}^{n-1}.$ Hence we can associate to $H$ the set $\left[
H\right]  =\left\{  \left[  k_{i}\right]  \right\}  _{i=1}^{n}\subset
\mathbb{PC}^{n-1}.$ Note that if we rescale $H$ to $\tilde{H}$ by selecting
scalars $\left\{  \lambda_{i}\right\}  $ declaring the $\left\{  \lambda
_{i}k_{i}\right\}  $ to be the kernels for $\tilde{H},$ then $[\tilde
{H}]=\left[  H\right]  .$ The rescaled space $\tilde{H}$, and in some sense
$\tilde{H}$ is the generic rescaling of $H$, produces the same set in
$\mathbb{PC}^{n-1}.$ From our point of view, our interest is in properties
invariant under rescaling, this is an attractive feature.

Next note that there set $\left[  H^{\#}\right]  $ associated to the dual RKHS
$H^{\#}$ can also be described in the language of $\mathbb{PC}^{n-1}.$ To
$n-1$ generic points in $\mathbb{PC}^{n-1}$ there correspond $n-1$ linearly
independent lines in $\mathbb{C}^{n}$. There is a unique line in
$\mathbb{C}^{n}$ orthogonal to those lines and that orthogonal line
corresponds to a point in $\mathbb{PC}^{n-1}$ which is "orthogonal" to each of
the $n-1$ points in the starting set. This process applied to each of the
$n-1$ point subsets of $\left[  H\right]  $ gives a set $\left[  H\right]
^{\#}$. Tracking the definitions we see that $\left[  H\right]  ^{\#}=\left[
H^{\#}\right]  .$

If $H$ is orthogonal then $\left\{  k_{i}\right\}  $ and $\{k_{i}^{\#}\}$ are
related through a period two conjugate linear isometry of the Hilbert space.
This isometry descends to an isometry of $\mathbb{PC}^{n-1}$ which
interchanges $\left[  H\right]  $ and $\left[  H^{\#}\right]  .$ On the other
hand if there is such an isometry connecting $\left[  H\right]  $ and $\left[
H^{\#}\right]  $ then one can show by analysis in $\mathbb{PC}^{n-1}$ that the
isometry extends to a global isometry of $\mathbb{PC}^{n-1}$ which, by
Wigner's theorem, must come from a unitary or antiunitary map of the original
$\mathbb{C}^{n}.$

In sum, many of the ideas we have considered can be comfortably reformulated
as statements about certain types of symmetric point sets in projective space.


\begin{thebibliography}{99999}                                                                                            %


\bibitem[AM]{AM}Agler, J. and McCarthy J. Pick Interpolation and Hilbert
Function Spaces, Graduate Studies in Mathematics, 44, 2002.

\bibitem[ARSW1]{ARSW1}Arcozzi, N. Rochberg, R. Sawyer, E. Wick, B. D. Distance
functions for reproducing kernel Hilbert spaces. Function spaces in modern
analysis, 25--53, Contemp. Math., 547, Amer. Math. Soc., Providence, RI, 2011.

\bibitem[ARSW2]{ARSW2}Arcozzi, N., Rochberg, R., Sawyer, E., Wick B., The
Dirichlet space and related function spaces. Mathematical Surveys and
Monographs, 239. American Mathematical Society, Providence, RI, 2019.

\bibitem[CLW]{CLW}Cole, B., Lewis, K., Wermer, J. A characterization of Pick
Bodies, J. Lond. Math. Soc 48 (1993) 316-328.

\bibitem[GMR]{GMR}Garcia, S. Mashreghi, J. Ross, W. Introduction to model
spaces and their operators, Cambridge Studies in Advanced Mathematics, 148.
Cambridge University Press, Cambridge, 2016.

\bibitem[GP]{GP}Garcia, S. Putinar, M. Complex symmetric operators and
applications. Trans. Amer. Math. Soc. 358 (2006), no. 3, 1285--1315.

\bibitem[Go]{Go}Goldman, W. Complex Hyperbolic Geometry, Okford Mathematical
Monographs, Oxford University Oressm 1999.

\bibitem[R1]{R1}Rochberg, R. Complex hyperbolic geometry and Hilbert spaces
with complete Pick kernels. J. Funct. Anal. 276 (2019), no. 5, 1622--1679.

\bibitem[R2]{R2}Rochberg, R. Tetrahedra in complex hyperbolic space and
Hilbert spaces with Pick kernels. arXiv:2003.10921, 2020

\bibitem[Sh]{Sh}Shalit, O, Operator theory and function theory in the
Drury-Arveson Space and its quotients. Operator Theory (2015) 1125-1180.
\end{thebibliography}
\end{document}